\documentclass[11pt]{amsart}
\usepackage{amsmath,amssymb}
\usepackage{verbatim}
\usepackage{color}
\usepackage{hyperref}

\begin{document}

\newtheorem{thm}{Theorem}[section]
\newtheorem{theorem}{Theorem}[section]
\newtheorem{lem}[thm]{Lemma}
\newtheorem{lemma}[thm]{Lemma}
\newtheorem{prop}[thm]{Proposition}
\newtheorem{proposition}[thm]{Proposition}
\newtheorem{corollary}[thm]{Corollary}
\newtheorem{definition}[thm]{Definition}
\newtheorem{remark}[thm]{Remark}
\newtheorem{conjecture}[theorem]{Conjecture}

\numberwithin{equation}{section}

\newcommand{\Z}{{\mathbb Z}} 
\newcommand{\Q}{{\mathbb Q}}
\newcommand{\R}{{\mathbb R}}
\newcommand{\C}{{\mathbb C}}
\newcommand{\N}{{\mathbb N}}
\newcommand{\FF}{{\mathbb F}}
\newcommand{\fq}{\mathbb{F}_q}
\newcommand{\X}{{\mathbb {X}}}
\newcommand{\rmk}[1]{\footnote{{\bf Comment:} #1}}

\newcommand{\bfA}{{\boldsymbol{A}}}
\newcommand{\bfY}{{\boldsymbol{Y}}}
\newcommand{\bfX}{{\boldsymbol{X}}}
\newcommand{\bfZ}{{\boldsymbol{Z}}}
\newcommand{\bfa}{{\boldsymbol{a}}}
\newcommand{\bfy}{{\boldsymbol{y}}}
\newcommand{\bfx}{{\boldsymbol{x}}}
\newcommand{\bfz}{{\boldsymbol{z}}}
\newcommand{\F}{\mathcal{F}}
\newcommand{\Gal}{\mathrm{Gal}}
\newcommand{\Fr}{\mathrm{Fr}}
\newcommand{\Hom}{\mathrm{Hom}}
\newcommand{\GL}{\mathrm{GL}}

\renewcommand{\mod}{\;\operatorname{mod}}
\newcommand{\ord}{\operatorname{ord}}
\newcommand{\TT}{\mathbb{T}}
\renewcommand{\i}{{\mathrm{i}}}
\renewcommand{\d}{{\mathrm{d}}}
\renewcommand{\^}{\widehat}
\newcommand{\HH}{\mathbb H}
\newcommand{\Vol}{\operatorname{vol}}
\newcommand{\area}{\operatorname{area}}
\newcommand{\tr}{\operatorname{tr}}
\newcommand{\norm}{\mathcal N} 
\newcommand{\intinf}{\int_{-\infty}^\infty}
\newcommand{\ave}[1]{\left\langle#1\right\rangle} 
\newcommand{\Var}{\operatorname{Var}}
\newcommand{\Prob}{\operatorname{Prob}}
\newcommand{\sym}{\operatorname{Sym}}
\newcommand{\disc}{\operatorname{disc}}
\newcommand{\CA}{{\mathcal C}_A}
\newcommand{\cond}{\operatorname{cond}} 
\newcommand{\lcm}{\operatorname{lcm}}
\newcommand{\Kl}{\operatorname{Kl}} 
\newcommand{\leg}[2]{\left( \frac{#1}{#2} \right)}  
\newcommand{\Li}{\operatorname{Li}}

\newcommand{\sumstar}{\sideset \and^{*} \to \sum}

\newcommand{\LL}{\mathcal L} 
\newcommand{\sumf}{\sum^\flat}
\newcommand{\Hgev}{\mathcal H_{2g+2,q}}
\newcommand{\USp}{\operatorname{USp}}
\newcommand{\conv}{*}
\newcommand{\dist} {\operatorname{dist}}
\newcommand{\CF}{c_0} 
\newcommand{\kerp}{\mathcal K}

\newcommand{\Cov}{\operatorname{cov}}
\newcommand{\Sym}{\operatorname{Sym}}

\newcommand{\ES}{\mathcal S} 
\newcommand{\EN}{\mathcal N} 
\newcommand{\EM}{\mathcal M} 
\newcommand{\Sc}{\operatorname{Sc}} 
\newcommand{\Ht}{\operatorname{Ht}}

\newcommand{\E}{\operatorname{E}} 
\newcommand{\sign}{\operatorname{sign}} 

\newcommand{\divid}{d} 
\newcommand{\A}{{\mathbb{A}}}
\newcommand{\h}{\mathbb{H}_{2g+1}}
\newcommand{\p}{\mathbb{P}_{2g+1}}
\newcommand{\f}{\mathbb{F}_{q}[T]}
\newcommand{\z}{\zeta_{\A}}
\newcommand{\lo}{\log_q}
\newcommand{\x}{\chi}
\newcommand{\xx}{\mathcal{X}}
\newcommand{\lL}{\mathcal{L}}
\newcommand{\e}{\varepsilon}
\newcommand{\w}{\omega}
\newcommand{\pp}{\text{\textbf{P}}}	
\newcommand{\M}{\mathcal{M}}

\title[HYBRID EULER-HADAMARD PRODUCT]
{Hybrid Euler-Hadamard Product for Dirichlet $L$-functions with Prime conductors over Function Fields}

\author{Julio Andrade}
\address{Department of Mathematics, University of Exeter, Exeter, EX4 4QF, United Kingdom}
\email{j.c.andrade@exeter.ac.uk}

\author{ASMAA sHAMESALDEEN}
\address{Department of Mathematics, University of Exeter, Exeter, EX4 4QF, United Kingdom}
\email{as1029@exeter.ac.uk}

\subjclass[2010]{Primary 11M38; Secondary 11M06, 11G20, 11M50, 14G10}
\keywords{Mean values of $L$--functions; finite fields; function fields; Euler-Hadamard Product}

\begin{abstract}
In this paper we extend the hybrid Euler-Hadamard product model for quadratic Dirichlet $L$-functions associated to irreducible polynomials over function fields. We also establish an asymptotic formula for the first twisted moment in this family of $L$-functions and then we provide further evidence for the conjectural asymptotic formulas for its moments. 
\end{abstract}
\date{\today}

\maketitle


\section{Introduction}\label{into}

There has been substantial and sustained research into moments of families of $L$-functions on the critical line. Much of this interest is generated by the presence of numerous applications, but  moments  are  also  studied  for  their  own  intrinsic  interest and for being a well-known challenging problem.

For the family of quadratic Dirichlet $L$-functions $L(s,\x_d)$, where $\x_d$ is real primitive Dirichlet character modulo $d$, Jutila \cite{Jutila} established the first and second moments and following improvements on the error terms were obtained by Goldfeld and Hoffstein \cite{Gold Hoff}, Soundararajan \cite{Sond} and Young \cite{Young}. In a breakthrough result, Soundararajan in \cite{Sond} obtained an asymptotic formula for the third moment. In a recent paper Diaconu and Whitehead \cite{DW} established a smooth asymptotic formula for the third moment and proved the existence of a secondary term of size $x^{3/4}$. Under the Generalized Riemann Hypothesis, Soundararajan and Young \cite{Sand Young} claimed that they could establish an asymptotic formula for the fourth moment. Recently, Shen \cite{Shen} obtained an asymptotic formula for the fourth moment of quadratic Dirichlet $L$-function at $s=1/2$ and a precise lower bound unconditionally. 

No other asymptotic formulas are known for the moments of quadratic Dirichlet $L$-functions at the centre of the critical strip. However, using the analogy with random matrix theory, Keating and Snaith \cite{Keating Snaith} have put forward a conjecture for the leading order asymptotic for all moments of quadratic Dirichlet $L$-function which agrees with the previous results. Conrey, Farmer, Keating, Rubinstein and Snaith introduced a different method, known as the recipe \cite{CFKRS}, and they were able to give conjectures for the moments beyond the leading order asymptotic to include all the principal lower order terms. Using different techniques, known as multiple Dirichlet series, Diaconu, Goldfeld and Hoffstein \cite{DGH} also produced conjectures for the moments of $L$-functions and their method also predicted the existence of many lower order terms for higher moments for this family of $L$-functions.

A similar problem was investigated by Goldfeld and Viola \cite{Goldfeld Viola} where they have conjectured an asymptotic formula for the moments of quadratic Dirichlet $L$-functions associated to primes $p\equiv 3 \mod 4$. In this context, Jutila \cite{Jutila} established an asymptotic formula for the first moment of this family of $L$-functions. Recently, under Generalized Riemann Hypothesis, Baluyot and Pratt \cite{BP} obtained the
leading order term in the asymptotic for the second moment in this family. 

In the function field setting, for the family of quadratic Dirichlet $L$-functions $L(s,\x_D)$, where $D$ is monic, square-free polynomial in $\f$, Hoffstein and Rosen \cite{H&R} computed the first moment of this family. Following Jutila's idea \cite{Jutila}, Andrade and Keating \cite{a&kmeanvalue} obtained an asymptotic formula for the first moment when $q$ is fixed and $q\equiv 1 \mod 4$, with improvements on secondary main term and a better error term by Florea \cite{Florea1}. The second, third and fourth moments were computed by Florea \cite{Florea2,Florea3}. Moreover, Diaconu \cite{Diaconu}  proved the existence of a secondary term in the asymptotic formula of the third moment of quadratic Dirichlet $L$-functions. In another paper, Andrade and Keating \cite{a&kConInMo} adapted the recipe of \cite{CFKRS} and \cite{Conr-Far-Zir} to the function field setting and conjectured asymptotic formulas for the
integral moments and ratios of the family of quadratic Dirichlet $L$-functions in function fields.

In this paper, we consider the family of quadratic Dirichlet $L$-functions associated to irreducible polynomials in $\mathbb{F}_q[T]$. Denote by $\p$ the space of monic, irreducible polynomials of degree $2g+1$ over $\f$. We are interested in studying the asymptotic behaviour for the $k^{th}$ moment,
\begin{equation*}
I_k(g)=\frac{1}{|\p|} \sum_{P\in\p} L\left(\tfrac{1}{2},\x_P\right)^k,
\end{equation*}
as $g\to\infty$ and $q$ is fixed and $\x_P(f)$ is given by the Legendre symbol in $\f$.

Andrade and Keating \cite{a&kprimemean} established an asymptotic formulas for $I_k(g)$ when $k=1,2$ and in a recent paper Bui and Florea \cite{BF2} improved the error term in the formula for $I_2(g)$ obtaining an extra lower order term. For other values of $k$, following the recipe to the function field setting \cite{a&kConInMo}, Andrade, Jung and Shamesaldeen \cite{Andrade Jung Shames} proposed a general formula for the integral moments of quadratic Dirichlet $L$-functions associated to $\x_P$ over function fields.   

\begin{conjecture}\label{julio and asmaa conjecture}
	For any $k\in\N$ we have 
	
	\begin{equation*}
	\frac{1}{\left|\p\right|} \sum_{P\in\p} L\left(\tfrac{1}{2},\x_P\right)^k \sim 2^{-\frac{k}{2}} \mathcal{A}_k \frac{G\left(k+1\right)\sqrt{\Gamma(k+1)}}{\sqrt{G\left(2k+1\right)\Gamma(2k+1)}} \left(2g\right)^{k(k+1)/2}
	\end{equation*}
	as $g\to\infty,$ where 
	
	\begin{equation}\label{A_k}
	\mathcal{A}_k = \prod_{\substack{Q \text{ monic} \\ \text{ irreducible} }} \left[ \left(1-\frac{1}{|Q|}\right)^{\frac{k(k+1)}{2}} \left(1+ \sum_{j=1}^\infty\frac{d_k\left(P^{2j}\right)}{|P|^j}\right)\right]
	\end{equation}
	with $d_k(f)$ being the $k^{th}$ divisor function, and $G(k)$ is the Barnes' $G$-function.
\end{conjecture}

\begin{remark}
	An equivalent form of $\mathcal{A}_k$ is 
	
	\begin{equation*}
	\mathcal{A}_k = \prod_{\substack{Q \text{ monic} \\ \text{ irreducible} }} \left(1-\frac{1}{|Q|}\right)^{\frac{k(k+1)}{2}} \left(\frac{1}{2} \left(1-\frac{1}{|P|^{1/2}}\right)^{-k} + \frac{1}{2} \left(1+\frac{1}{|P|^{1/2}}\right)^{-k}\right) 
	\end{equation*}
\end{remark}

Another method to produce conjectures for the moments of $L$-functions is through the hybrid Euler-Hadamard product. It was proven by Gonek, Hughes and Keating \cite{GHK}, using a smoothed form of the explicit formula of Bombieri and Hejhal \cite{Bom Hej}, that the value of Riemann zeta function at a height $t$ on the critical line can be approximated by a partial Euler product multiplied by a partial Hadamard product over the non-trivial zeros close to $1/2+it$. The value distribution of the partial Hadamard product is expected to be modelled by the characteristic polynomial of a large random unitary matrix, since it involves only local information about the zeros. Calculating the moments of the partial Euler product rigorously and making an assumption (which can be proved in certain cases) about the independence of the two products, Gonek, Hughes and Keating were able to reproduce the conjecture for the $2k^{th}$ moment of the Riemann zeta-function. Bui and Keating \cite{Bui Keating 1,Bui Keating 2} extended this approach to the moments of Dirichlet $L$-function $L(s,\x)$ at the centre of the critical strip $(s=1/2)$.

In the function field setting, Bui and Florea \cite{BF} gave further support for Andrade and Keating conjecture by using the idea of Gonek, Hughes and Keating and the Hybrid-Euller formula for $L(s,\x_D)$. They also evaluated the first three twisted moments of quadratic Dirichlet $L$-functions over function fields for the family of square-free monic polynomials. 

In this paper, we give further support for Conjecture \ref{julio and asmaa conjecture} and derive the first twisted moment of quadratic Dirichlet $L$-functions associated to $\x_P$ over function fields. The recent work of Bui and Florea {BF2} can be used to obtain the twisted second moment of this family and this is being investigated by the authors.


\section{Statement of Results}\label{statement of results}

In this paper we assume $q$ is fixed and $q\equiv 1(\mod 4)$. Let $\mathbb{F}_q$ be a finite field with $q$ elements, and $\A=\f$ be the polynomial ring over $\mathbb{F}_q$. We denote $\M$ to be the set of all monic polynomials in $\A$, $\M_n$ be the set of those of degree $n$ and $\M_{\leqslant n}$ of degree at most $n$. The monic irreducible polynomial over $\A$ will be denoted by $\mathbb{P}$, and the set of all monic irreducible polynomials of degree $n$ is denoted by $\mathbb{P}_n.$ For a polynomial $f\in\A$, we denote its norm by $|f|$ which is defined to be $q^{\deg(f)}$ if $f\neq 0$ and 0 otherwise. The von Mangoldt function is defined by 
\begin{equation*}
\Lambda(f) = \begin{cases}
\deg(P) & \text{ if }f=cP^i \text{ for some } c\in\mathbb{F}^\times_q \text{ and } i\geqslant 1\\
0 & \text{ otherwise.}
\end{cases}
\end{equation*} 

Note that, from the Polynomial Prime Theorem \cite[Theorem 2.2]{Rosen} we have
\begin{equation}\label{PPT}
|\p|\sim\frac{q^{2g+1}}{2g+1}.
\end{equation}
We define the expected value for any function $F$ on $\p$ by 
\begin{equation*}
\left<F\right>_{\p}:=\frac{1}{|\p|} \sum_{P\in\p}F(P).
\end{equation*}

The Euler-Hadamard product for the quadratic $L$-functions associated to $P$ is stated below.

\begin{theorem}\label{Hybrid Euler-Hadamard Product}
	Let $X$ be a positive real number and $u(x)$ be real, non-negative, $C^\infty$-function with mass 1 and compactly supported on $\left[q,q^{1+1/X}\right]$. Let 
	
	\begin{equation}\label{U}
	U(z)=\int_0^\infty u(t)E_1\left(z\log t\right) dt,
	\end{equation}
	where $E_1(z)$ is the exponential integral, $E_1(z)=\int_z^\infty e^{-t}/t dt.$ Then for $\mathfrak{R}(s)\geqslant 0$ we have 
	
	\begin{equation*}
	L(s,\x_P) = P_X(s,\x_P) Z_X(s,\x_P),
	\end{equation*}
	where
	
	\begin{equation*}
	P_X(s,\x_P)= \exp\left(\sum_{\substack{f \text{ monic} \\ \deg(f)\leqslant X}} \frac{\Lambda(f)\x_P(f)}{|f|^s \deg(f)}\right)
	\end{equation*}
	and
	
	\begin{equation}\label{Z_X(s)}
	Z_X(s,\x_P)=\exp\left(- \sum_\rho U\left((s-\rho)X\right)\right),
	\end{equation}
	where the sum is over all the zeros $\rho$ of $L(s,\x_P).$ 
\end{theorem}

\begin{remark}
	The proof of the Theorem \ref{Hybrid Euler-Hadamard Product} is similar to \cite[Theorem 2.1]{BF}.
\end{remark}

\begin{remark}
As discussed in \cite{GHK}, $P_X(s,\x_P)$ can be thought of as the Euler product for $L(s,\x_P)$ truncated to include polynomials of degree $\leqslant X$, and $Z_X(s,\x_P)$ can be thought of as the Hadamard product for $L(s,\x_P)$ truncated to include zeros within a distance $\lesssim 1/X$ from the point $s$. The parameter $X$ thus controls the relative contributions of the Euler and Hadamard products. Note that a similar hybrid product formula was developed independently by Andrade, Keating and Gonek in \cite{AGK} and Bui and Florea in \cite{BF}.
\end{remark}

In Section \ref{P_X} we compute the moments of $P_X(\x_P):=P_X(1/2,\x_P)$ and prove the following theorem. 

\begin{theorem}\label{partial Euler product}
	Let $0<c<2.$ Suppose $X\leqslant(2-c)\log g/\log q.$ Then for any $k\in\R$ we have 
	
	\begin{equation*}
	\left<P_X(\x_P)^k\right>_{\p} = 2^{-k/2} \mathcal{A}_k \left(e^\gamma X\right)^{k(k+1)/2}+ O_k\left(q^{-X/2}X^{k(k+1)/2-1}\right),
	\end{equation*} 	
	where $\mathcal{A}_k$ is defined as in Conjecture \ref{julio and asmaa conjecture}.
\end{theorem}

As in Bui and Florea \cite{BF}, for the partial Hadamard product, $Z_X(\x_P):=Z_X\left(\tfrac{1}{2},\x_P\right),$ we conjecture that 

\begin{conjecture}\label{201}
	For any $k\in\N$ we have 
	\begin{equation*}
	\left< Z_X\left(\x_P\right)^k\right>_{\p} \sim \frac{G\left(k+1\right)\sqrt{\Gamma(k+1)}}{\sqrt{G\left(2k+1\right)\Gamma(2k+1)}} \left(\frac{2g}{e^\gamma X}\right)^{\frac{k(k+1)}{2}}
	\end{equation*}
\end{conjecture}  

We note from Theorem \ref{Hybrid Euler-Hadamard Product}  that $L(1/2,\x_P)P_X(\x_P)^{-1} = Z_X(\x_P)$. This allow us to derive the first moment of $Z_X.$

\begin{theorem}\label{Z_X}
	Let $0<c<2$. Suppose $X\leqslant(2-c)\log g/\log q.$ Then we have 
	\begin{equation*}
	\left< L\left(\tfrac{1}{2},\x_P\right) P_X(\x_P)^{-1}\right>_{\p} = \frac{2g}{e^{\gamma} X} + O\left(g X^{-2}\right). 
	\end{equation*}
\end{theorem}

Theorem \ref{partial Euler product} and Theorem \ref{Z_X} suggest that when $X$ is not very large relative to $q^g$, the $k$th moment of $L(1/2,\x_P)$ is asymptotic to the product of the moments of $P_X(\x_P)$ and $Z_X(\x_P)$ for $k=1$.  In general we believe that this is true and we present the following conjecture. 

\begin{conjecture}(Splitting Conjecture). 
	Let $0<c<2$. Suppose $X\leqslant (2-c)\log g/\log q$ and $X,g $ tends to $\infty$. Then for any $k\geqslant 0$ we have 

\begin{equation*}
\left< L\left(\tfrac{1}{2},\x_P\right)^k\right>_{\p} \sim \left<P_X(\x_P)^k\right>_{\p} \left<Z_X(\x_P)^k\right>_{\p}.
\end{equation*}
\end{conjecture}

One can note that Theorem \ref{partial Euler product}, Conjecture \ref{201} and the Splitting Conjecture imply Conjecture \ref{julio and asmaa conjecture}. 

Before proving Theorem \ref{Z_X} we have to have knowledge and understanding about the twisted moments of quadratic Dirichlet $L$-functions over function fields, i.e.
\begin{equation*}
I(l;g)= \left<L\left(\frac{1}{2},\x_P\right) \x_P(l)\right>_{\p}.
\end{equation*}
Throughout this paper we assume that $l$ is a monic polynomial in $\f$ with $\deg(l)\ll g$. In Section \ref{1twisted} we shall establish the following result.

\begin{theorem}(Twisted first moment).\label{1TWM}
	Let $l=l_1l_2^2$ with $l_1$ square-free. Then we have 
	\begin{equation*}
	\begin{split}
	I(l;g) &= \frac{1}{|l_1|^{1/2}} \left(g-\deg(l_1)+1\right) + O\left(  |l_1|^{-\frac{1}{2}} |l|^{\frac{1-\varepsilon}{2}} q^{\frac{3}{2}(\varepsilon-1) g} \right) \\
	& \ \ \ + O\left(q^{-\frac{g}{2}} \left(\deg(l)+g\right)\right).
	\end{split}
	\end{equation*}	
\end{theorem}


 \section{Preparations}
 
We first give some preliminary facts about quadratic Dirichlet $L$-functions in function fields.
		
	
\subsection{Facts about $\f$}
${\color{white}gg}$\\

The zeta function of $\f$ is denoted by $\z(s)$ and defined by
\begin{equation*}
\begin{split}
\z(s) &:=\sum_{f \text{ monic}} \frac{1}{|f|^s} = \prod_{\substack{P \text{ monic} \\ \text{irreducible}}} \left(1-|P|^{-s}\right)^{-1}, \ \ \ \text{Re}(s)>1.
\end{split}
\end{equation*} 
Since there are $q^n$ monic polynomials of degree $n$, it can easily be shown that 
\begin{equation*}
\z(s)=\frac{1}{1-q^{1-s}},
\end{equation*}
 which provides an analytic continuation of the zeta-function to the whole complex plane, with simple pole at $s=1$ and no zeros. This fact leads to the analogue of Prime Number Theorem for polynomials in $\f$.
 
 \begin{theorem}(Prime Polynomial Theorem)\\
 	If $\pi_{\A}(n)$ denotes the number of monic irreducible polynomials in $\A$ of degree $n$, then
 	 \begin{equation}\label{PNT}
 	 \pi_ {\A}(n)=\frac{q^n}{n}+O\left(\frac{q^{n/2}}{n}\right).
 	 \end{equation}	
 \end{theorem}
		
\subsection{Quadratic Dirichlet $L$-functions for $\x_P$}
${\color{white}gg}$\\
	
For $P$ a monic irreducible polynomial, define the quadratic character $\left(\frac{f}{P}\right)$ by 
\begin{equation*}
\left(\frac{f}{P}\right)= \begin{cases}
1 & \text{ if } f \text{ is a square (mod }P),P\nmid f\\
-1 & \text{ if } f \text{ is not a square (mod }P),P\nmid f\\
0  & \text{ if } P\mid f.\\
\end{cases}
\end{equation*}
The quadratic reciprocity law states that for $A,B$ non-zeros and relatively prime monic polynomials, we have 
\begin{equation*}
\left(\frac{A}{B}\right)= \left(\frac{B}{A}\right) \left(-1\right)^{\frac{q-1}{2} \text{deg}(A)\text{deg}(B)}. 
\end{equation*}
We denote by $\x_P$ to be the quadratic character defined in terms of the quadratic residue symbol for $A$
\begin{equation*}
\x_P(f)=\left(\frac{P}{f}\right),
\end{equation*}
where $f\in \f$. 

In this paper, the focus will be in the family of quadratic Dirichlet $L$-functions associated to polynomials $P\in\p.$ The quadratic Dirichlet $L$-function attached to the character $\x_P$ is defined to be 
\begin{equation*}
\begin{split}
L\left(s,\x_P\right) &:= \sum_{f\in\M} \frac{\x_P(f)}{|f|^s}  = \prod_{\substack{Q \text{ monic} \\ \text{ irreducible} }} \left(1-\frac{\x_P(Q)}{|Q|^{s}}\right)^{-1},\text{\color{white} fjd} \text{Re}(s)>1.
\end{split} 
\end{equation*}
From \cite[Propositions 4.3, 14.6 and 17.7]{Rosen}, $L(s,\x_P)$ is a polynomial in $u=q^{-s}$ of degree $\deg(P)-1$ and 
\begin{equation*}
L(s,\x_P)=\mathcal{L}(u,\x_P) = L_{C_P}(u),
\end{equation*}
where $L_{C_P}(u)$ is the numerator of the zeta function associated to the hyperelliptic curve given in affine form by 
\begin{equation*}
C_P:y^2=P(T)
\end{equation*}
with
\begin{equation*}
P(T)=T^{2g+1}+a_{2g}T^{2g}+\cdots+a_1T+a_0
\end{equation*}
a monic irreducible polynomial in $\f$ of degree $2g+1$. From Weil \cite{weil} we know that $L(s,\chi_{P})$ satisfies a Riemann Hypothesis.

\subsection{Preliminary Lemmas}
${\color{white}gg}$\\

The first Lemma in this section is from \textup{\cite[Lemma 3.3]{a&kmeanvalue}}.
\begin{lemma}
	We have the following ``approximate" functional equation
	\begin{equation}\label{aprox}
	L\left(\tfrac{1}{2},\x_P\right) = \sum_{f\in\M_{\leqslant g}} \frac{\x_P(f)}{\sqrt{|f|}} + \sum_{f\in\M_{\leqslant g-1}} \frac{\x_P(f)}{\sqrt{|f|}}.
	\end{equation}
\end{lemma}

\begin{lemma}\label{expected_value} For any monic polynomial $f\in\f$ we have
	\begin{equation*}
	\lim_{\deg(P) \to \infty} \frac{1}{|\p|} \sum_{P\in\p}\x_P(f)= 
	\begin{cases} 
	1 &\mbox{if } f=\Box \\ 
	0 & \mbox{otherwise. } \end{cases}
	\end{equation*}
\end{lemma}
\begin{proof}
	Consider the case when $f=\Box,$ then we have 
	
	\begin{equation*}
	\sum_{P\in\p} \x_P(f)=\sum_{P\in\p}\x_P(l^2) = \sum_{\substack{P\in\p \\ P\nmid l}} 1, 
	\end{equation*}
	since we are summing over primes of degree $2g+1$ and $P\nmid l,$ and $\text{deg}(l)\leqslant 2g,$ which means that we are counting all primes of degree $2g+1,$ thus
	
	\begin{equation*}
	\sum_{\substack{P\in \p \\ P\nmid l}} 1 = |\p|.
	\end{equation*}
	Hence if $n$ is a square of a polynomial,
	
	\begin{equation*}
	\lim_{\deg(P)\to\infty} \frac{1}{|\p|} \sum_{P\in\p} \x_P(f) = 1.
	\end{equation*}
	
	It remains to consider the case when $f\neq\Box.$ Rudnick \cite{Rudnick} has proven that
	
	\begin{equation}\label{chi}
	\left|\sum_{P\in\p}\x_P(f)\right| \ll \frac{q^g}{g} \text{deg}(f),
	\end{equation}
	and from Prime Polynomial Theorem \ref{PNT} we have
	
	\begin{equation*}
	\begin{split}
	\frac{1}{|\p|} \sum_{P\in\p} \x_P(f) & \ll q^{-g} \text{deg}(f).
	\end{split}
	\end{equation*}
	Hence if $n$ is not a square of a polynomial we have that 
	
	\begin{equation*}
	\lim_{\deg(P)\to\infty} \frac{1}{|\p|} \sum_{P\in\p} \x_P(f) = 0.
	\end{equation*}
\end{proof}

The next two lemmas are quoted from \cite{BF}.
\begin{lemma}\label{Mertens} (Merten's Theorem)
	We have 
	
	\begin{equation*}
	\prod_{\substack{P \text{ irreducible} \\ \deg(P)\leqslant X}} \left(1-\frac{1}{|P|}\right)^{-1} = e^\gamma X+O(1),
	\end{equation*}
	where $\gamma$ is the Euler constant.
	\end{lemma}

\begin{lemma}\label{perron} (Perron's Formula)\\
	If the power series 
	
	\begin{equation*}
	H(u)=\sum_{f \in\M} a(f) u^{\deg(f)}
	\end{equation*}
	converges absolutely for $|u|\leqslant R< 1$, then 
	
	\begin{equation*}
	\sum_{f\in\M_n} a(f)= \frac{1}{2\pi i}\oint_{|u|=R} \frac{H(u)}{u^{n+1}} du,
	\end{equation*}
	and 
	
	\begin{equation*}
	\sum_{f\in\M_{\leqslant n}} a(f)= \frac{1}{2\pi i}\oint_{|u|=R} \frac{H(u)}{(1-u)u^{n+1}} du.
	\end{equation*}
\end{lemma}


\section{Moments of the partial Euler product}\label{P_X}

Recall that 

\begin{equation*}
P_X(s,\x_P)= \exp\left(\sum_{\substack{f\in \M_{\leqslant X}}} \frac{\Lambda(f)\x_P(f)}{|f|^s\deg(f)}\right),
\end{equation*}
and let 

\begin{equation*}
\begin{split}
P^*_{k,X}&(s,\x_P)\\
&= \prod_{\substack{ \deg(Q)\leqslant X/2}} \left(1-\frac{\x_P(Q)}{|Q|^s}\right)^{-k} \prod_{\substack{ X/2<\deg(Q)\leqslant X}} \left(1+\frac{k\x_P(Q)}{|Q|^s}+\frac{k^2\x_P(Q)^2}{2|Q|^{2s}} \right)
\end{split}
\end{equation*}
for any $k\in\R$. Consider the following lemma \cite[Lemma 5.1]{BF}.

\begin{lemma}\label{P*&P}
	For any $k\in\R$ we have 	 
	
	\begin{equation*}
	P_X(s,\x_P)^k = P_{k,X}^*(s,\x_P) \left(1+O\left(q^{-X/6}/X\right)\right)
	\end{equation*}
	uniformly for Re$(s)=\sigma\geqslant 1/2.$
\end{lemma}

Now we proceed with the proof of Theorem \ref{partial Euler product}, and write $P_{k,X}^*(s,\x_P)$ as a Dirichlet series   

\begin{equation*}
\begin{split}
\sum_{f \in\M} & \frac{\alpha_k(f)\x_P(f)}{|f|^s} \\
&= \prod_{\substack{ \deg(Q)\leqslant X/2}} \left(1-\frac{\x_P(Q)}{|Q|^s}\right)^{-k} \prod_{\substack{ X/2<\deg(Q)\leqslant X}} \left(1+\frac{k\x_P(Q)}{|Q|^s}+\frac{k^2\x_P(Q)^2}{2|Q|^{2s}} \right)
\end{split}
\end{equation*} 
with $\alpha_k(f)\in\R.$ If we denote the set of $X$-smooth polynomials by $S(X)$, that is, 
\begin{equation*}
S(X)=\left\{f\in \f : \text{ monic, } Q\mid f \to \deg(Q) \leqslant X \right\},
\end{equation*}
then $\alpha_k(f)$ is multiplicative, and vanishes when $f \notin S(X).$ We also have $0\leqslant\alpha_k(f)\leqslant d_{|k|}(f)$ for all $f\in \M$. Moreover, $\alpha_k(f)=d_k(f)$ if $f\in S(X/2),$ and $\alpha_k(Q)=k$ and $\alpha_k(Q^2)=k^2/2$ for all prime $Q\in\f$ with $X/2<\deg(Q)\leqslant X.$

Truncating the series, for $s=1/2$, at $\deg(f)\leqslant \vartheta g,$ where $\vartheta>0$ will be chosen later. Using the Prime Polynomial Theorem (\ref{PPT}) we can bound the following sum by

\begin{equation}\label{3}
\begin{split}
\sum_{\substack{f\in S(X) \\ \deg(f)>\vartheta g}} \frac{\alpha_k(f)\x_P(f)}{|f|^{1/2}} &\leqslant \sum_{f\in S(X)} \frac{d_{|k|}(f)}{|f|^{1/2}} \left(\frac{|f|}{q^{\vartheta g}}\right)^{c/4}\\
&= q^{-\vartheta g/4} \prod_{\substack{ \deg(Q)\leqslant X}}\left(1-\frac{1}{|Q|^{(2-c)/4}}\right)^{-|k|} \\
&\ll q^{-c\vartheta g/4} \exp\left(O_k\left(\sum_{\substack{Q \text{ irreducible} \\ \deg(Q)\leqslant X}}\frac{1}{|Q|^{(2-c)/4}}\right)\right) \\
&\ll q^{-c\vartheta g/4} \exp\left(O_k\left( \frac{q^{(2+c)X/4}}{X}\right)\right) \\
&\ll_{k,\varepsilon} q^{-c\vartheta g/4+\varepsilon g},
\end{split}
\end{equation}
since $X\leqslant (2-c)\log g/\log q.$ Hence we have 
\begin{equation}\label{P*}
P_{k,X}^*(\x_P) :=P_{k,X}^*\left(\tfrac{1}{2},\x_P\right) =  \sum_{\substack{f\in S(X) \\ \deg(f)\leqslant\vartheta g}} \frac{\alpha_k(f)\x_P(f)}{|f|^{1/2}} +O_{k,\varepsilon} \left(q^{-c\vartheta g/4+\varepsilon g}\right)
\end{equation}
for all $k\in\R$ and $\vartheta>0,$ and in the end we obtain that

\begin{equation*}
\begin{split}
\frac{1}{\#\p} \sum_{P\in \p}& P_X^*(\x_P)^k \\
 = & \frac{1}{\#\p}  \sum_{\substack{f\in S(X) \\ \deg(f)\leqslant\vartheta g}} \frac{\alpha_k(f)}{|f|^{1/2}} \sum_{P\in \p}\x_P(f)+O_{k,\varepsilon} \left(q^{-c\vartheta g/4+\varepsilon g}\right).
\end{split}
\end{equation*}

We consider the main term in the above equation and break it into two sums, $I_{(f\neq\Box)}$ and $I_{(f=\Box)}$ with $f\neq\Box$ and $f=\Box$ respectively. First we bound the contribution of the terms with $f\neq\Box$ by using Rudnick's inequality (\ref{chi})
\begin{equation*}
\begin{split}
I_{(f\neq\Box)} \ll q^{-g}\sum_{\substack{f\in S(X)}} \frac{d_{|k|}(f)\deg(f)}{|f|^{1/2}}.
\end{split}
\end{equation*}
Following the same argument as in equation (\ref{3}),
\begin{equation*}
\begin{split}
\sum_{\substack{f\in S(X)}} \frac{d_{|k|}(f)\deg(f)}{|f|^{1/2}} &\ll \prod_{ \deg(Q)\leqslant X} \left(1-\frac{\deg(Q)}{|Q|^{1/2}}\right)^{-|k|}\\
& \ll_k \exp\left(O_k\left( q^{\frac{X}{2}} \right) \right) \\
& \ll_{k,\varepsilon} q^{\varepsilon g}.
\end{split}
\end{equation*}
Thus
\begin{equation*}
\begin{split}
I_{(f\neq\Box)} \ll q^{-g+\varepsilon g}.
\end{split}
\end{equation*}

Now we consider the contribution of $f=\Box.$ Using Lemma \ref{expected_value} and the fact that $0\leqslant\alpha_k(f)\leqslant d_{|k|}(f)$ we can write

\begin{equation*}
\begin{split}
I_{(f=\Box)} &= \sum_{\substack{f\in S(X) \\ \deg(f)\leqslant\vartheta g/2}} \frac{\alpha_k(f^2)}{|f|}.
\end{split}
\end{equation*}
As in (\ref{3}), we can bound the sum over all $f\in S(X)$ with $\deg(f)>\vartheta g$ by $q^{-\vartheta g/4+\varepsilon g},$ and extend the sum to all $f\in S(X),$
\begin{equation*}
\begin{split}
&I_{(f=\Box)}=  \sum_{\substack{f\in S(X)}} \frac{\alpha_k(f^2)}{|f|} + O_{k,\varepsilon}\left(q^{-\vartheta g/4+\varepsilon g}\right).
\end{split}
\end{equation*}
Then, by the multiplicativity of $\alpha_k(f)$ and Lemma \ref{Mertens},
\begin{equation}\label{box}
\begin{split}
& I_{(f=\Box)} = \prod_{\substack{\deg(Q)\leqslant X}}\left(1+\sum_{j=1}^\infty \frac{\alpha_k(Q^{2j})}{|Q|^j}\right) + O_{k,\varepsilon}\left(q^{-\vartheta g/4+\varepsilon g}\right) \\
&= \prod_{\substack{\deg(Q)\leqslant X/2}}\left(1+\sum_{j=1}^\infty \frac{d_k(Q^{2j})}{|Q|^j}\right)  \prod_{\substack{X/2<\deg(Q)\leqslant X}}\left(1+ \frac{k^2}{2|Q|}+O_k\left( |Q|^{-2} \right) \right)\\
& \ \ \ \ + O_{k,\varepsilon}\left(q^{-\vartheta g/4+\varepsilon g}\right)\\
&= \left(1+ O_k\left(\frac{q^{-X/2}}{X}\right)\right) \prod_{\substack{\deg(Q)\leqslant X/2}}\left[ \left(1-\frac{1}{|Q|}\right)^{\frac{k(k+1)}{2}}\left(1+\sum_{j=1}^\infty \frac{d_k(Q^{2j})}{|Q|^j}\right)\right] 
\\
&\ \ \ \ \times \prod_{\substack{\deg(Q)\leqslant X/2}} \left(1-\frac{1}{|Q|}\right)^{-\frac{k(k+1)}{2}} \prod_{\substack{ X/2<\deg(Q)\leqslant X}}\left(1-\frac{1}{|Q|} \right)^{-\frac{k^2}{2}} + O_{k,\varepsilon}\left(q^{-\vartheta g/4+\varepsilon g}\right)\\
\end{split}
\end{equation}

\begin{equation*}
\begin{split}
&= \left(1+  O_k\left(\frac{q^{-X/2}}{X}\right)\right) \; 2^{-\frac{k}{2}} \mathcal{A}_k \left(e^\gamma X\right)^{\frac{k(k+1)}{2}} + O_{k,\varepsilon}\left(q^{-\vartheta g/4+\varepsilon g}\right).\\
&\\
\end{split}
\end{equation*}

Hence, 
\begin{equation*}
\begin{split}
\frac{1}{|\p|} \sum_{P\in\p} P^*_X(\x_P)^k = 2^{-\frac{k}{2}} \mathcal{A}_k \left(e^{\gamma}X\right)^{\frac{k(k+1)}{2}} + O_k\left(q^{-X/2}X^{\frac{k(k+1)}{2}-1}\right)
\end{split}
\end{equation*}
which finishes the proof of Theorem \ref{partial Euler product}. \;\;\;\; \; \; \;\;\;\; \; \; \;\;\;\; \; \; \;\;\;\; \; \; \;\;\;\; \; $\Box$


\section{Twisted First Moment of $L(\frac{1}{2},\x_P)$}\label{1twisted}

The aim in this section is to evaluate the first twisted moment 

\begin{equation*}
I(l;g) = \frac{1}{|\p|}\sum_{P\in\p} L\left(\tfrac{1}{2},\x_P\right)\x_P(l),
\end{equation*}
using the assumption that $l\in\M$ in $\f$ with $\deg(l)\ll g.$ Using the ``approximate" functional equation (\ref{aprox}) write

\begin{equation*}
I(l;g) = S(l;g)+S(l;g-1),
\end{equation*}
where 
\begin{equation*}
S(l;N)=  \frac{1}{|\p|} \sum_{f\in\M_{\leqslant N}} \frac{1}{|f|^{1/2}}\sum_{P\in\p}\x_P(fl)
\end{equation*}
for $N\in\{g,g-1\}$. According to whether the degree of the product $fl$ is odd or even, respectively, we have

\begin{equation*}
S(l;N)=  S^o(l;N)+S^e(l;N).
\end{equation*}
Consider $S^o(l;N)$, since $\deg(fl)$ is odd then $fl$ is not a perfect square and using Rudnick's inequality (\ref{chi}) we have the bound 

\begin{equation*}
\begin{split}
S^o(l;N) &\ll q^{-g}  \sum_{f \in \M_{\leqslant N}} \frac{1}{|f|^{1/2}}  \deg(fl)\\
&\ll q^{-g+N/2}\left(\deg(l)+N \right).
\end{split}
\end{equation*}
Therefore, 

\begin{equation*}
\begin{split}
S^o(l) &= S^o(l;g) + S^o(l;g-1)\\
& \ll q^{-\frac{g}{2}} \left(\deg(l)+g \right).
\end{split}
\end{equation*}

Now, it remains to compute the main term which comes from evaluating $S^e(l;N).$ Writing $f=f_1^2l_1$ and $l=l_1l_2^2$ where $f_1,l_1,l_2\in\M$ and $l_1$ square-free we have

\begin{equation*}
\begin{split}
S^e(l;N) &= \frac{1}{|\p|}  \sum_{\substack{f_1 \in \M\\ \deg(f_1^2l_1)\leqslant N}} \frac{1}{|f_1^2l_1|^{1/2}}   \sum_{\substack{P\in\p}}\x_P((f_1l_1l_2)^2).
\end{split}
\end{equation*}
Note that, $\x_P(f^2)=1$ for all $f\in\M$, and that the degree of $f_1l_1l_2$ might be greater than the degree of $P\in \p,$ therefore we need to divide the $S(l;N)$ into two sums 

\begin{equation*}
\begin{split}
S^e(l;N) = S^e_1(l;N) + S^e_2(l;N),
\end{split}
\end{equation*}
where

\begin{equation*}
S^e_1(l;N) =  \frac{1}{|l_1|^{1/2}} \frac{1}{|\p|}  \sum_{f_1 \in \M_{( N-\deg(l_1))/2}} \frac{1}{|f_1|} \sum_{\substack{P\in\p \\ P\nmid f_1l_1l_2 \\ \deg(f_1l_1l_2)< 2g+1}} \x_P((f_1l_1l_2)^2) 
\end{equation*}
and

\begin{equation*}
S^e_2(l;N) =  \frac{1}{|l_1|^{1/2}} \frac{1}{|\p|} \sum_{f_1 \in \M_{( N-\deg(l_1))/2}} \frac{1}{|f_1|} \sum_{\substack{P\in\p \\ P\nmid f_1l_1l_2 \\ \deg(f_1l_1l_2)\geqslant 2g+1}} \x_P((f_1l_1l_2)^2).
\end{equation*}

Applying the Prime Polynomial Theorem (\ref{PNT}) for $S_1^e(l;N)$ we have

\begin{equation*}
\begin{split}
&S^e_1(l;N)\\
&= \frac{1}{|l_1|^{1/2}} \sum_{2n\leqslant N-\deg(l_1)} \sum_{f_1 \in \M_{n}} \frac{1}{|f_1|} + O\left(\frac{q^{-g}}{|l_1|^{1/2}} \sum_{2n\leqslant N-\deg(l_1)} \sum_{f_1 \in \M_{n}} \frac{1}{|f_1|} \right).
\end{split}
\end{equation*}

Using Perron's formula [Lemma \ref{perron}], we obtain the following for the main term

\begin{equation*}
\begin{split}
&S_1^e(l;N)\\
&= \frac{1}{|l_1|^{1/2}} \frac{1}{2\pi\i}\oint_{|u|=r}  \frac{\mathcal{F}(u)}{u^{N-\deg(l_1)+1}(1-u)}du+O\left( \frac{q^{-g}}{|l_1|^{1/2}} \left[\frac{N-\deg(l_1)}{2}\right] \right)
\end{split}
\end{equation*}
for any $r<1$, where $\mathcal{F}(u)$ is multiplicative and defined by 

\begin{equation*}
\begin{split}
\mathcal{F}(u) &= \sum_{f_1 \in \M}\frac{1}{|f_1|} u^{2\deg(f)}\\
& = \mathcal{Z}\left(\frac{u^2}{q}\right)\\
&= \frac{1}{(1-u)(1+u)}.
\end{split}
\end{equation*}

Thus,
\begin{equation*}
\begin{split}
S_1^e(l;N) &=  \frac{1}{|l_1|^{1/2}} \frac{1}{2\pi\i}\oint_{|u|=r}  \frac{1}{u^{N-\deg(l_1)+1}(1-u)^2(1+u)}du\\
& \ \ \ +O\left( \frac{q^{-g}}{|l_1|^{1/2}} \left[\frac{N-\deg(l_1)}{2}\right] \right),
\end{split}
\end{equation*}
and 

\begin{equation*}
\begin{split}
S_1^e(l) &= S_1^e(l;g)+ S_1^e(l;g-1)\\
&=  \frac{1}{|l_1|^{1/2}} \frac{1}{2\pi\i}\oint_{|u|=r}  \frac{1}{u^{g-\deg(l_1)+1}(1-u)^2}du  + O\left(\frac{q^{-g}}{|l_1|^{1/2}} \left(g-\deg(l_1)\right)\right).\\
\end{split}
\end{equation*}
It is clear that the integrand's numerator has an analytic continuation to the region $|u|\leqslant R_1= q^{1-\varepsilon}.$ Therefore, we move the contour of integration to $|u|=R_1,$ encountering a pole of order $2$ at $u=1$. The integral over the contour $|u|=R_1$ is bounded by $|l_1|^{1/2} q^{-g+\varepsilon g}.$ Then, using the residue theorem we have
\begin{equation*}
\begin{split}
S_1^e(l) &= -\frac{1}{|l_1|^{1/2}} \text{ Res}\left(u=1\right) +  O\left(|l_1|^{\frac{1}{2}}q^{-g+\varepsilon g} \right)  + O\left(q^{-g}|l_1|^{-\frac{1}{2}} \left(g-\deg(l_1)\right)\right)\\
&=  \frac{1}{|l_1|^{1/2}} \left(g-\deg(l_1)+1\right) + O\left(|l_1|^{\frac{1}{2}}q^{-g+\varepsilon g} \right).
\end{split}
\end{equation*}

It remains to bound $S^e_2(l;N)$. As in (\ref{3}) we extend the sum over $P\in\p$ and using the Prime Polynomial Theorem (\ref{PPT}) we have
\begin{equation*}
\begin{split}
S^e_2(l;N) &\ll  \frac{1}{|l_1|^{1/2}} \frac{1}{|\p|} \sum_{f_1 \in \M_{( N-\deg(l_1))/2}} \frac{1}{|f_1|} \sum_{\substack{P\in\p}} \left(\frac{|f_1l_1l_2|}{q^{2g+1}}\right)^{1-\varepsilon}\\
&\ll q^{-2g+2\varepsilon g} |l_1|^{-\frac{\varepsilon}{2}} |l|^{\frac{1-\varepsilon}{2}}  \sum_{f_1 \in \M_{( N-\deg(l_1))/2}} \frac{1}{|f_1|^\varepsilon} \\
&\ll |l_1|^{-\frac{1}{2}} |l|^{\frac{1-\varepsilon}{2}} q^{-2g+2\varepsilon g} q^{(1-\varepsilon)\frac{N}{2}}.
\end{split} 
\end{equation*}
Therefore,

\begin{equation*}
\begin{split}
S^e_2(l) &= S^e_2(l;g)+S^e_2(l;g-1)\\
&\ll   |l_1|^{-\frac{1}{2}} |l|^{\frac{1-\varepsilon}{2}} q^{\frac{3}{2}(\varepsilon-1) g}.
\end{split}
\end{equation*}
Combining the results in this section we conclude the proof of Theorem \ref{1TWM}.
{\color{white}g} \;\;\;\; \; \; \;\;\;\; \; \; \;\;\;\; \; \; \;\;\;\; \; \; \;\;\;\; \; \;\;\;\; \; \;\;\;\; \; \;\;\;\; \; \;\;\;\; \; \; \;\;\;\; \; \; \;\;\;\; \; \; \;\;\;\; \; \; \;\;\;\; \; $\Box$


\section{First Moment of the Partial Hadamard Product}\label{FTM}

\subsection{Random Matrix Theory Model}
$\color{white}gf$\\
From (\ref{Z_X(s)}) recall that 
\begin{equation*}
Z_X(s,\x_P)=\exp\left(-\sum_{\rho}U\left(\left(s-\rho\right)X\right)\right),
\end{equation*}
where 
\begin{equation*}
U(z)\int_0^\infty u(x) E_1\left(z\log x\right)dx.
\end{equation*}
Denote the zeros by $\rho=1/2+i\gamma.$ Since $E_1(-ix)+E_1(ix)=-2$Ci$\left(|x|\right)$ for $x\in \R$, where Ci$(z)$ is the cosine integral,
\begin{equation*}
\text{Ci}(z)=-\int_z^\infty \frac{\cos(x)}{x}dx,
\end{equation*}
then we have 
\begin{equation}\label{23}
\left<Z_X(\x_P)^k\right>_{\p}= \left< \prod_{\gamma>0} \exp\left(2k\int_0^\infty u(x)\text{Ci}\left(\gamma X\left(\log x\right)\right)dx\right)\right>_{\p}.
\end{equation}
We model the right hand side of (\ref{23}) by replacing the ordinates $\gamma$ by the eigenangles of a $2g\times 2g$ symplectic unitary matrix and averaging over all such matrices with respect to the Haar measure. The $k^{th}$-moment of $Z_X(\x_P)$ is thus expected to be asymptotic to 
\begin{equation*}
\mathbb{E}_{2g}\left[\prod_{n=1}^g \exp\left(2k\int_0^\infty u(x)\text{Ci}\left(\theta X\left(\log x\right)\right)dx\right)\right],
\end{equation*}
where $\pm\theta_n$ with $0\leqslant\theta_1\leqslant\cdots\leqslant\theta_g\leqslant\pi$ are the $2g$ eigenangles of the random matrix and $\mathbb{E}_{2g}[\cdot]$ denotes the expectation with respect to the Haar measure. It is convenient to have our function periodic, therefore we instead consider
\begin{equation}\label{24}
\mathbb{E}_{2g}\left[\prod_{n=1}^g \phi\left(\theta_n\right)\right]
\end{equation}
where
\begin{equation*}
\begin{split}
\phi(\theta) &=\exp\left(2k \int_0^\infty u(x) \left(\sum_{j=-\infty}^\infty\text{Ci}\left(|\theta+2\pi j|X\left(\log x\right)\right)\right)dx\right)\\
&= \left|2\sin\frac{\theta}{2}\right|^{2k}\exp\Bigg(2k \int_0^\infty u(x) \left(\sum_{j=-\infty}^\infty\text{Ci}\left(|\theta+2\pi j|X\left(\log x\right)\right)\right)dx \\
& \ \ \   -2k\log\left|2\log\sin\frac{\theta}{2}\right|\Bigg).
\end{split}
\end{equation*}
The average (\ref{24}) over the symplectic group has been asymptotically evaluated in \cite{DIK} and so we have 
\begin{equation*}
\mathbb{E}_{2g}\left[\prod_{n=1}^g\phi(\theta_n)\right]\sim \frac{G(k+1)\sqrt{\Gamma(k+1)}}{\sqrt{G(2k+1)\Gamma(2k+1)}} \left(\frac{2g}{e^\gamma X}\right)^{k(k+1)/2}.
\end{equation*}

\subsection{Proof of Theorem \ref{Z_X}}
$\color{white}gf$\\

Now we are in a position to evaluate the first moment of $Z_X(\x_P)$. From Theorem \ref{1TWM}

\begin{equation*}
\begin{split}
\left< L\left(\tfrac{1}{2},\x_P\right)\x_P(l)\right>_{\p} =& \frac{1}{|l_1|^{1/2}} g + O\left(\frac{\deg(l_1)}{|l_1|^{1/2}}\right) + O\left(  |l_1|^{-\frac{1}{2}} |l|^{\frac{1-\varepsilon}{2}} q^{\frac{3}{2}(\varepsilon-1) g} \right).
\end{split}
\end{equation*}
Combining with (\ref{P*}) we get

\begin{equation*}
\begin{split}
\left< L\left(\tfrac{1}{2},\x_P\right) P_X^{* -1}(\x_P) \right>_{\p} 
=& J_1 + J_2 + O\left(q^{\vartheta g - \frac{3}{2}g +\frac{3}{2}\varepsilon g} \right) + O\left(q^{-c\vartheta g/4+\varepsilon g}\right).
\end{split}
\end{equation*}
for any $\vartheta>0$, where 

\begin{equation*}
J_1 = g \sum_{\substack{l\in S(X) \\ \deg(l)\leqslant\vartheta g}} \frac{\alpha_{-1}(l)}{|l|^{1/2}} \frac{1}{|l_1|^{1/2}} 
\end{equation*}
and 
\begin{equation*}
J_2 = \sum_{\substack{l\in S(X) \\ \deg(l)\leqslant\vartheta g}} \frac{\alpha_{-1}(l)}{|l|^{1/2}} \frac{1}{|l_1|^{1/2}} \deg(l_1). 
\end{equation*}
Thus Theorem \ref{Z_X} follows from evaluating $J_1$ and bounding $J_2$ and choosing any $0<\vartheta<3/2$. 

\subsubsection{Evaluating $J_1$} 
\begin{equation}\label{J1}
\begin{split}
J_1 &=  g \sum_{\substack{l_1,l_2\in S(X) \\ l_1 \text{ square-free} \\\deg(l_1)+2\deg(l_2) \leqslant\vartheta g}} \frac{\alpha_{-1}(l_1l_2^2)}{|l_1||l_2|}.
\end{split}
\end{equation}

Recall that the function $\alpha_{-1}(l)$ is given by 

\begin{equation} \label{alpha -1}
\begin{split}
\sum_{\substack{l \text{ monic}\\ \deg(l)= n}} &\frac{\alpha_{-1}(l) \x_P(l)}{l^s}\\
&  = \prod_{\substack{Q \text{ irreducible} \\ \deg(Q)\leqslant X/2}} \left(1-\frac{\x_P(Q)}{|Q|^s}\right) \prod_{\substack{Q \text{ irreducible} \\ X/2<\deg(Q)\leqslant X}} \left(1- \frac{\x_P(Q)}{|Q|^s}+\frac{\x_P(Q)^2}{2|Q|^{2s}}\right).  
\end{split}
\end{equation}
So $\alpha_{-1}(l)$ is supported on square-free polynomials. Thus, if we let 

\begin{equation*}
\mathcal{Q}_X = \prod_{\substack{Q \text{ irreducible} \\ \deg(Q)\leqslant X}} Q,
\end{equation*}
then for the sum over $l_1,l_2$ in (\ref{J1}) we can write $l_1=l_1'l_3$ and $l_2=l_2'l_3$, where $l_1',l_2',l_3$ are all square-free, that is, $l_1',l_2',l_3|\mathcal{Q}_X,$ and $l_1',l_2',l_3$ are pairwise co-prime. Thus

\begin{equation*}
\begin{split}
J_1 = g &\sum_{\substack{l_3|\mathcal{Q}_X \\ deg(l_3)\leqslant \vartheta g/3}} \frac{\alpha_{-1}(l_3^3)}{|l_3|^2} \sum_{\substack{l_2|(\mathcal{Q}_X/l_3) \\ deg(l_2)\leqslant (\vartheta g-3\deg(l_3))/2}} \frac{\alpha_{-1}(l_2^2)}{|l_3|} \\
& \times \sum_{\substack{l_1|(\mathcal{Q}_X/l_2l_3) \\ deg(l_1)\leqslant \vartheta g- 2 \deg(l_2)-3\deg(l_3)}} \frac{\alpha_{-1}(l_1)}{|l_1|}.
\end{split}
\end{equation*}

As in (\ref{3}) we can remove the condition $\deg(l_1)+2\deg(l_2)+3\deg(l_3)\leqslant \vartheta g$ at the cost of an error of size $O_\varepsilon\left(q^{-\vartheta g/2+\varepsilon g}\right)$. Now we define the following multiplicative functions

\begin{equation*}
\begin{split}
T_1(f) = \sum_{l|f} \frac{\alpha_{-1}(l)}{|l|}, \; T_2(f) = \sum_{l|f} \frac{\alpha_{-1}(l^2)}{|l|T_1(l)} \text{ and } T_3(f) = \sum_{l|f} \frac{\alpha_{-1}(l^3)}{|l|^2T_1(l)T_2(l)}.
\end{split}
\end{equation*}
Then

\begin{equation*}
J_1 = g \prod_{\substack{Q \text{ irreducible} \\ \deg(Q)\leqslant X}} \left(1 + \frac{\alpha_{-1}(Q)}{|Q|} +\frac{\alpha_{-1}(Q^2)}{|Q|} + \frac{\alpha_{-1}(Q^3)}{|Q|^2}\right).
\end{equation*}

We note from (\ref{alpha -1}) that $\alpha_{-1}(Q)=-1$ and when $\deg(Q)\leqslant X/2$, we have $\alpha_{-1}(Q^2)=0$ and $\alpha_{-1}(Q^3)=0$ and when $x/2< \deg(Q)\leqslant X$  we have $\alpha_{-1}(Q^2)=1/2$ and $\alpha_{-1}(Q^3)=0$. Thus,

\begin{equation*}
\begin{split}
J_1 & = g \prod_{\substack{Q \text{ irreducible} \\ \deg(Q)\leqslant X/2}} \left(1 - \frac{1}{|Q|}\right)  \prod_{\substack{Q \text{ irreducible} \\ X/2<\deg(Q)\leqslant X}} \left(1 - \frac{1}{|Q|} +\frac{1}{2|Q|^2} \right)\\
&= g \frac{1}{e^{\gamma}X/2}  + O\left(g X^{-2}\right) 
\end{split}
\end{equation*}

\subsubsection{Bounding $J_2$}

We have that

\begin{equation}\label{j_2}
\begin{split}
J_2 &\ll  \sum_{\substack{l_1,l_2 \in S(X) }} \frac{\alpha_{-1}(l_1l_2^2)}{|l_1l_2^2|^{1/2}} \frac{1}{|l_1|^{1/2}} \deg(l_1)\\
& \ll \sum_{\substack{l_1 \in S(X)}} \frac{1}{|l_1|} \deg(l_1) \sum_{\substack{l_2 \in S(X)}} \frac{1}{|l_2|}, 
\end{split}
\end{equation}
since $0\leqslant\alpha_k(f)\leqslant T_{|k|}(f).$ Let 

\begin{equation*}
\begin{split}
F(\sigma) &= \sum_{l\in S(X)} \frac{1}{|l|^\sigma}\\
&= \prod_{\substack{Q \text{ irreducible} \\ \deg(Q)\le X}} \left(\sum_{j=0}^\infty \frac{1}{|Q|^{j\sigma}}\right).
\end{split}
\end{equation*}
Then by Merten's Theorem [Lemma \ref{Mertens}]
$$F(1) \asymp \prod_{\substack{Q \text{ irreducible} \\ \deg(Q)\leqslant X}} \left(1-\frac{1}{|P|}\right)^{-1} \asymp X,$$
and 

\begin{equation*}
\begin{split}
F'(\sigma) &= -\sum_{l\in S(X)} \frac{1}{|l_1|} \deg(l_1)  \\
& = F(\sigma)\log q \sum_{\substack{Q \text{ irreducible} \\ \deg(Q)\leqslant X}} \frac{\sum_{j=0}^\infty j \deg(Q)/|Q|^{j\sigma}}{\sum_{j=0}^\infty1/|Q|^{j\sigma}}. 
\end{split}
\end{equation*}
We note that the sum over $l_1$ in (\ref{j_2}) is  

\begin{equation*}
\begin{split}
-\frac{F(1)}{\log q} &= F(1) \sum_{\substack{Q \text{ irreducible} \\ \deg(Q)\leqslant X}} \frac{\deg(Q)}{1-|Q|},
\end{split}
\end{equation*}
and hence it is

\begin{equation*}
\begin{split}
&\ll X \sum_{\substack{Q \text{ irreducible} \\ \deg(Q)\leqslant X}} \frac{\deg(Q)}{|Q|}\\
& \ll X^2.
\end{split}
\end{equation*}

We can see that the sum over $l_2$ in (\ref{j_2}) is $F(1)$ and therefore is bounded by $X$. Hence

\begin{equation*}
J_2 \ll X^3
\end{equation*} 
and this concludes the proof. \;\;\;\;\;\;\;\;\;\;\;\;\; \; \;\;\;\; \; \; \;\;\;\; \; \; \;\;\;\; \; \; \;\;\;\; \; \; \;\;\;\; \; $\Box$


\vspace{0.5cm}

\noindent \textbf{Acknowledgment:} The first author is grateful to the Leverhulme Trust (RPG-2017-320) for the support through the research project grant “Moments of L-functions in Function Fields and Random Matrix Theory”. The second author was supported by a Ph.D. scholarship from the government of Kuwait.


\end{document}